\documentclass[11pt]{article}
\usepackage[top=2.5cm,bottom=2.5cm,left=2.5cm,right=2.5cm]{geometry}
\usepackage{mathrsfs}
\usepackage{pstricks}
\usepackage{amsmath,amssymb,graphicx,amsthm,tabularx,array}
\usepackage{hyperref}
\usepackage{makecell}

\theoremstyle{plain}
\newtheorem{theorem}{Theorem}[section]

\newtheorem{coro}[theorem]{Corollary}
\newtheorem{lem}[theorem]{Lemma}

\newtheorem{myCli}{Claim}

\newtheorem{conj}[theorem]{Conjecture}

\newtheorem{definition}[theorem]{Definition}

\theoremstyle{definition}

\newtheorem{other}{}

\title{{A  solution to a  conjecture on the signless Laplacian spectral radius for
$t$-color-critical graphs}\thanks{
Ming-Zhu Chen is partly supported by the National Natural Science Foundation of China (No. 12461063),
Ya-Lei Jin is partly supported by the National Natural Science Foundation of China (Nos. 12371349 and 12471331),
Peng-Li Zhang is partly supported by the National Natural Science Foundation of China (No. 12501493).
The authors   are listed in alphabetical order.
E-mail addresses: mzchen@hainanu.edu.cn (M.-Z. Chen), yaleijin@shnu.edu.cn ($^\dag$Y.-L. Jin, corresponding author), zpengli0626@163.com (P.-L. Zhang),  zhengj@yangtzeu.edu.cn (J. Zheng).}}

\author{Ming-Zhu Chen$^a$, Ya-Lei Jin$^b\dag$, Peng-Li Zhang$^c$,Jian Zheng$^d$\\
{\small  $^a$ School of Mathematics and Statistics, Key Laboratory of Engineering Modeling and Statistical } \\
{\small  Computation of Hainan Province,
Hainan University, Haikou 570228, P.R. China}\\
{\small  $^b$ Department of Mathematics, Shanghai Normal University, Shanghai 200234, P.R. China}\\
{\small $^c$ School of Statistics and Data Science, Shanghai University of International} \\
{\small Business and Economics, Shanghai 201620, P.R. China}\\
{\small $^d$ School of  Information and Mathematics, Yangtze University, Jingzhou 434023,  P.R. China}\\
}

\date{}
\begin{document}
\maketitle

\begin{abstract}
An induced matching is a matching that forms an induced subgraph. A graph is
$t$-color-critical if removing some induced matching of size $t$ lowers its
chromatic number, but removing any $t-1$ vertices does not. Let $F$ be a
$t$-color-critical graph with $\chi(F)=r+1$. For sufficiently large $n$,
Simonovits determined the unique edge-extremal $F$-free graph on $n$ vertices.
Recently, Zheng, Li and Li [Linear Algebra Appl.\ 730 (2026) 546--565]
conjectured that, for $t\ge 2$ and $r\ge 3$, the join $K_{t-1}\vee T_{n-t+1,r}$
uniquely maximizes the signless Laplacian spectral radius among all $n$-vertex
$F$-free graphs when $n$ is sufficiently large. In this paper,
we prove this
conjecture.
In contrast to the usual spectral arguments, our proof of this conjecture relies on two techniques of a rather different flavour.
Our first technique is an analogue of Zykov symmetrization for the signless Laplacian matrix.
Our second technique is an induction on $n$, from which we  obtain the lower bound on the smallest entry of the Perron vector of a signless Laplacian spectral extremal graph  rather than a
structural statement.
\end{abstract}



 \section{Introduction}
A graph $G$ consists of a vertex set $V(G)$ and an edge set $E(G)$, where each
edge is a $2$-element subset of $V(G)$. We write $e(G):=|E(G)|$ for the number
of edges in $G$. A \emph{matching} in $G$ is a set $M\subseteq E(G)$ of pairwise
disjoint edges. For a positive integer $t$, a \emph{$t$-matching} is a matching
of size $t$. The \emph{matching number} of $G$, denoted by $\mu(G)$, is the
maximum cardinality of a matching in $G$. For positive integers
$a_1,\ldots,a_r$, let $K_r(a_1,\ldots,a_r)$ denote the complete $r$-partite
graph with partite sets of sizes $a_1,\ldots,a_r$. The \emph{Tur\'an graph}
$T_{n,r}$ is the complete $r$-partite graph on $n$ vertices whose partite sets
are as balanced as possible.
For two vertex-disjoint graphs $G_1,G_2$, their \emph{union}, denoted by
$G_1\cup G_2$, has vertex set $V(G_1)\cup V(G_2)$ and edge set
$E(G_1)\cup E(G_2)$. The \emph{join} $G_1\vee G_2$ is the graph obtained from
$G_1\cup G_2$ by adding all edges between $V(G_1)$ and $V(G_2)$.
Denote by $K_{n}$ and $C_{n}$  the complete graph and the cycle of order $n$, respectively. Denote by $\overline{G}$  the  complement of  a graph $G.$

A proper vertex coloring of a graph $H$ is an assignment of colors to the
vertices such that adjacent vertices receive distinct colors. The
\emph{chromatic number} $\chi(H)$ is the minimum number of colors in such a
coloring. For a vertex subset $U\subseteq V(H)$, the subgraph induced by $U$,
denoted by $H[U]$, has vertex set $U$ and edge set $\{e\in E(H): e\subseteq U\}$.
An \emph{induced matching} is a matching whose edges form an induced subgraph.

For a graph $F$, a graph $G$ is said to be \emph{$F$-free} if it contains no
copy of $F$ as a subgraph. The \emph{Tur\'an number} of $F$, denoted by
$\mathrm{ex}(n,F)$, is the maximum number of edges among all $F$-free graphs on
$n$ vertices. We write $\mathrm{EX}(n,F)$ for the set of all $n$-vertex
$F$-free graphs attaining this maximum. The classical Tur\'an's
theorem~\cite{Turan41} states that $\mathrm{EX}(n,K_{r+1})=\{T_{n,r}\}$. The
\emph{Tur\'an density} of $F$, denoted by $\pi(F)$, is defined as
\[
\pi(F):=\lim_{n\to\infty}\frac{\mathrm{ex}(n,F)}{\binom{n}{2}}.
\]
The celebrated Erd\H{o}s--Stone--Simonovits theorem~\cite{ES1946,ES1966} asserts
that $\pi(F)=1-\frac{1}{\chi(F)-1}$.

\begin{definition}\label{s-color-critical}
Let $F$ be a graph with $\chi(F)=r+1$, and let $t\geq 1$
be an integer. We call $F$ \emph{$t$-color-critical} if
the following conditions hold:
\begin{enumerate}
    \item[(i)]
    For every $S\subseteq V(F)$ with $|S|=t-1$, we have
    \[
        \chi(F-S)=r+1.
    \]

    \item[(ii)]
    There exists a partition
    \[
        V(F)=U_1\cup\cdots\cup U_r
    \]
    such that $F[U_1]$ consists of $t$ pairwise disjoint
    edges and possibly some isolated vertices, while
    $U_2,\ldots,U_r$ are independent sets.
\end{enumerate}
\end{definition}
 It is worth highlighting that this family includes a
wide range of graphs, such as color-critical graphs, the Petersen graph, Kneser graphs and disjoint unions of cliques.  In particular, a $1$-color-critical graph coincides with a
color-critical graph in the usual sense.

In 1974, Simonovits~\cite{Sim1974} determined the unique extremal graph
for $t$-color-critical graphs.

\begin{theorem}[\cite{Sim1974}]\label{thm:Simonovits}
Let $t\ge 1$, $r\ge 2$ and $F$ be a $t$-color-critical graph with
$\chi(F)=r+1$. If $n$ is sufficiently large and $G$ is an $n$-vertex $F$-free
graph, then
\[
e(G)\le e\bigl(K_{t-1}\vee T_{n-t+1,r}\bigr).
\]
Moreover, the equality holds if and only if $G\cong K_{t-1}\vee T_{n-t+1,r}$.
\end{theorem}

The adjacency matrix of a graph $G$ is defined as $A(G)=(a_{ij})_{n\times n}$,
where $a_{ij}=1$ if $ij\in E(G)$ and $a_{ij}=0$ otherwise. The \emph{adjacency
spectral radius} $\lambda(G)$ is the largest eigenvalue of $A(G)$. Determining
$\lambda(G)$ for $F$-free graphs is referred to as a \emph{spectral
Tur\'an-type problem}. Accordingly, we define $\mathrm{ex}_{\lambda}(n,F)$ as
the maximum adjacency spectral radius among all $n$-vertex $F$-free graphs, and
$\mathrm{EX}_{\lambda}(n,F)$ as the set of $n$-vertex $F$-free graphs attaining
this maximum. Guiduli \cite{Gui1996} and Nikiforov~\cite{Niki2007laa2}
independently established the spectral version of Tur\'an's theorem. In 2009,
Nikiforov~\cite{N2009} extended this result and proved that
$\mathrm{EX}_{\lambda}(n,F)=\{T_{n,r}\}$ for any color-critical graph $F$ with
$\chi(F)=r+1$. For further details, we refer the reader to the comprehensive
survey~\cite{LiLF}.

Fang et al.~\cite{Fang2026} and Zhang \cite{Zhang2026}    independently established the
adjacency spectral version of Theorem~\ref{thm:Simonovits}.

\begin{theorem}[\cite{Fang2026,Zhang2026}]\label{thm:adjacency}
Let $t\ge 1$, $r\ge 2$ and $F$ be a $t$-color-critical graph with
$\chi(F)=r+1$. If $n$ is sufficiently large, then
\[
\mathrm{EX}_{\lambda}(n,F)=\bigl\{K_{t-1}\vee T_{n-t+1,r}\bigr\}.
\]
\end{theorem}

The \emph{signless Laplacian matrix} of a graph $G$ is given by
$Q(G)=D(G)+A(G)$, where $D(G)$ is the degree diagonal matrix and $A(G)$ is the
adjacency matrix. The largest eigenvalue of $Q(G)$, denoted by $q(G)$, is
referred to as the \emph{signless Laplacian spectral radius} or the \emph{$Q$-index}
of $G$. Let $\mathrm{ex}_{q}(n,F)$ be the maximum $Q$-index among all
$F$-free graphs on $n$ vertices, and let $\mathrm{EX}_{q}(n,F)$ denote the
family of extremal graphs attaining this value.

In 2013, Freitas, Nikiforov and Patuzzi~\cite{FNP2013} initiated the
Tur\'an-type extremal problem for the signless Laplacian spectral radius:
determine the maximum possible $Q$-index of an $F$-free graph of order $n$.
Over the past decade, this problem has attracted considerable attention.  
 He, Jin and Zhang~\cite{HJZ2013} proved that every complete bipartite graph $K_{a,n-a}$ ($1\le a<n$) has the maximal signless Laplacian spectral radius for triangle-free graphs of order $n.$ They also proved that, for $r\geq 3$, the Tur\'{a}n graph $T_{n,r}$ is the unique signless Laplacian spectral extremal graph among $K_{r+1}$-free graphs of order $n$. Note that  $K_{r+1}$ is a color-critical graph of chromatic number $r+1.$ Later,
for color-critical graphs $F$ with $\chi(F)=r+1\geq4$, Zheng, Li and Li~\cite{ZLL2025} proved that $T_{n,r}$ is the
unique signless Laplacian spectral extremal graph for sufficiently large $n$,
which implies that $\mathrm{EX}_{q}(n,F)=\mathrm{EX}(n,F).$
For more related results, readers are referred to references~\cite{AFNP2016,CJZ2025,CLZ2020,CLZ2024,CWZ2022,CZ2021,LMX2022,NY2,
WZ2023,
Y2014,ZHG2021,ZLL2025}.
A recent theorem of Chen, Jin, Zhang and Zhang~\cite{CJZZ2026} provides a
more general connection between the Tur\'{a}n-type extremal problem and the signless Laplacian spectral Tur\'{a}n-type extremal problem: if $\mathrm{ex}(n,F)=e(T_{n,r})+O(1)$
with $\chi(F)=r+1\ge 4$, then $\mathrm{EX}_{q}(n,F)\subseteq \mathrm{EX}(n,F)$ for sufficiently large
$n$. For fixed $t\ge 2$,  Theorem~\ref{thm:Simonovits} together with the identity
\[
e(K_{t-1}\vee T_{n-t+1,r})-e(T_{n,r})=\frac{t-1}{r}\,n+O(1)
\]
shows that $t$-color-critical graphs lie outside the scope of this
bounded-error result. Nevertheless, Zheng, Li and Li~\cite{ZLL2025}
conjectured that $\mathrm{EX}_{q}(n,F)=\mathrm{EX}(n,F)$ for every $t$-color-critical graph $F$ and all sufficiently large $n$.


\begin{conj}[\cite{ZLL2025}]\label{conj:ZLL}
Let $t\ge 2$, $r\ge 3$ and  $F$ be a $t$-color-critical graph with
$\chi(F)=r+1$. If $n$ is sufficiently large and $G$ is an $n$-vertex $F$-free
graph, then
\[
q(G)\le q\bigl(K_{t-1}\vee T_{n-t+1,r}\bigr).
\]
Moreover, the equality holds if and only if $G\cong K_{t-1}\vee T_{n-t+1,r}$.
\end{conj}

Recently, Zhao, You and Zeng~\cite{ZYZ2026} proved Conjecture~\ref{conj:ZLL}  for
 $F=tK_{r+1}$ and   $t\ge 2$, $r\ge 3$.
In this paper, we confirm this
conjecture as follows.
\begin{theorem}\label{thm:main}
Let $t\ge 1$, $r\ge 3$ and  $F$ be a $t$-color-critical graph with
$\chi(F)=r+1$. If $n$ is sufficiently large and $G$ is an $n$-vertex $F$-free
graph, then
\[
q(G)\le q\bigl(K_{t-1}\vee T_{n-t+1,r}\bigr).
\]
Moreover, the equality holds if and only if $G\cong K_{t-1}\vee T_{n-t+1,r}$.
\end{theorem}

If $F$ is an $s$-vertex $t$-color-critical graph with $\chi(F)=r+1$, then
$F\subseteq \bigl(tK_2\cup \overline{K}_{s-2t}\bigr)\vee T_{s(r-1),\,r-1}$. In fact, we prove the following stronger result.
\begin{theorem}\label{thm:main1}
For integers $1\le t\le s/2$ and $r\ge 3$, let
\[
F=\bigl(tK_2\cup \overline{K}_{s-2t}\bigr)\vee T_{s(r-1),\,r-1}.
\]
Then, for sufficiently large $n$, $K_{t-1}\vee T_{n-t+1,r}$ is the unique graph
in $\mathrm{EX}_{q}(n,F)$.
\end{theorem}
Obviously, by Theorem~\ref{thm:main1}, Theorem~\ref{thm:main} holds.

For an integer $k\ge 2$,
recall that the $k$-th power $G^{k}$ of a graph $G$ joins two distinct vertices whenever
their distance in $G$ is at most $k$.
The following result is an immediate consequence of Theorem~\ref{thm:main1}.

\begin{coro}\label{cor:power}
Assume that $s=p(k+1)+h$, where $k,p\ge 2$ and $1\le h\le k$ are integers.
Let
\[
t=h-\Bigl\lfloor\frac{h-1}{p}\Bigr\rfloor p,
\qquad
r=k+\Bigl\lceil\frac{h}{p}\Bigr\rceil.
\]
Then, for sufficiently large $n$, $K_{t-1}\vee T_{n-t+1,r}$ is the unique graph
in $\mathrm{EX}_{q}(n,C_s^{k})$.
\end{coro}

\begin{proof}
As shown in the proof of \cite[Theorem~1.7]{Zhang2026},
\[
C_s^{k}\subseteq \bigl(tK_2\cup \overline{K}_{s-2t}\bigr)\vee T_{s(r-1),\,r-1},
\]
and $K_{t-1}\vee T_{n-t+1,r}$ is $C_s^{k}$-free. Since
$r=k+\bigl\lceil h/p\bigr\rceil\ge 3$, the desired result follows from
Theorem~\ref{thm:main1}.
\end{proof}


\noindent $\mathbf{Our ~ approach.}$ In the proof of Theorem~\ref{thm:main1}, we use two important tools.
Our first tool is  an analogue of Zykov symmetrization for the signless Laplacian (see Lemma~\ref{lem:2.2}).  It gives a Perron vector criterion: whenever the criterion holds, deleting all edges at a vertex $v$ and reconnecting $v$ to the neighbours of a non-adjacent vertex $u$ strictly increases $q$. Since the increase is strict,  the signless Laplacian spectral extremal   graph admits no symmetrizable pair. Thus the problem of maximizing the $Q$-index is transformed into a combinatorial structural problem for graphs.   Compared with the adjacency-matrix setting, the signless Laplacian symmetrization criterion developed in this paper needs to account simultaneously for changes in both the adjacency terms and the degree terms.


Our second tool is an induction on $n$,  based on Lemmas~\ref{min-comp} and~\ref{sdf}. The induction yields a  spectral statement rather than a structural one; specifically, we obtain the following lower bound on the smallest entry of the Perron vector of a signless Laplacian spectral extremal graph:
\begin{equation}\label{ast1}
x_{\min}^{2} \ge \frac1n\Bigl(1-\frac{1}{\log\log n}\Bigr).
\end{equation}
Lemma~\ref{min-comp}
asserts that (\ref{ast1}) holds for infinitely many orders, and Lemma~\ref{pro-lem}
uses (\ref{ast1}) to determine the exact structure of the extremal graph. Hence Theorem~\ref{thm:main1}   holds for  some sufficiently large integer $n_0$, which starts the induction. Now assume  Theorem~\ref{thm:main1}  holds for $n-1$. Then
$\mathrm{ex}_{q}(n-1,F)=q\bigl(K_{t-1}\vee T_{n-t,r}\bigr)$, which gives an upper bound on the ratio $\mathrm{ex}_{q}(n-1,F)/\mathrm{ex}_{q}(n,F)$. At the same time, applying
Lemma~\ref{sdf} to a vertex $u$ with minimum Perron entry in a signless Laplacian spectral extremal graph
$G\in\mathrm{EX}_{q}(n,F)$ gives a lower bound on the same ratio. The two bounds
agree to within an error of $O(n^{-2})$, which re-establishes (\ref{ast1}) at order $n$.
Lemma~\ref{pro-lem} then yields $G\cong K_{t-1}\vee T_{n-t+1,r}$, completing the induction.

\paragraph{Organization.}
 The rest of the paper is structured as follows. In Section $2$, we collect the necessary preliminaries. Section $3$ is devoted to the proof of Lemma~\ref{pro-lem}. In Section $4$, we prove Theorem~\ref{thm:main1}.

\section{Preliminaries}

Let $G$ be a graph.
For a vertex $v\in V(G)$, the \emph{neighborhood} $N_G(v)$ denotes the set of vertices adjacent to $v$,
and the \emph{degree} $d_G(v):=|N_G(v)|$ is the number of such vertices.
Let $\delta(G)$ denote the \emph{minimum degree} of $G$.
For a vertex subset $U\subseteq V(G)$,
we write $N_U(v):=N_G(v)\cap U$ for the set of neighbors of $v$ in $U$,
$d_U(v):=|N_U(v)|$ for its cardinality.
For $V_1,V_2\subseteq V(G)$, let $E_G(V_1,V_2):=\{uv\in E(G):u\in V_1,\ v\in V_2\}$
denote the set of edges between $V_1$ and $V_2$,
and write $e_G(V_1,V_2):=|E_G(V_1,V_2)|$.
When the graph $G$ is clear from the context, we abbreviate $e_G(V_1,V_2)$ to $e(V_1,V_2)$.   We say that distinct vertices
$v_1 , v_2 , ..., v_\tau$ of $G$ are symmetric if there are no edges among ${v_1,v_2,...,v_\tau}$ and $N_{G}(v_1)=N_{G}(v_i)$ for all $1\leq i\leq \tau.$

\begin{lem}[\cite{CFTZ2022}]\label{dx}
Let $V_{1},\ldots, V_{n}$ be $n$ finite sets. Then $|V_{1}\cap\cdots\cap V_{n}|\geq\sum_{i=1}^{n}
|V_{i}|-(n-1)|\cup_{i=1}^{n}V_{i}|$.
\end{lem}

In \cite{ZLS2025}, Zheng, Li and Su established an asymptotic value of $\mathrm{ex}_{q}(n,F)$ for any graph $F$. In particular,
for a graph $F$ with $\chi(F)\geq 3$, the value $\mathrm{ex}_{q}(n,F)$ admits the following  expression.

\begin{theorem}[\cite{ZLS2025}]\label{ess}
If  $F$ is a graph with chromatic number $\chi(F)\geq 3$, then
$$\mathrm{ex}_q (n,F)=(2\pi(F)+o(1))n.$$
\end{theorem}

For a vector $\mathbf{x}=(x_{1},\ldots,x_{n})$, we denote  $x_{\textup{min}}=\min\{x_{1},\ldots,x_{n}\}$.

\begin{lem}[\cite{AN2013}]\label{min}
Let $G$ be a graph of order $n$ with $q(G)=q$ and minimum degree $\delta(G)=\delta$. If  $\mathbf{x}=(x_{1},\ldots,x_{n})$ is a nonnegative unit  eigenvector  corresponding to $q$, then
$$x_{\textup{min}}^{2}(q^{2}-2q\delta+n\delta)\leq \delta.$$
\end{lem}

\begin{lem}\label{mindu}
 Let  $F$ be a graph with $\chi(F)\geq3$. Let $G\in \mathrm{EX}_{q}(n,F)$, and $\mathbf{x}=(x_{1},\ldots,x_{n})$ be a nonnegative unit  eigenvector  corresponding  to $q(G)$. If $0<\varepsilon<1$ and $x_{\textup{min}}^{2}>\frac{1-\varepsilon}{n}$, then for sufficiently large $n$,
$$\delta(G)>(\pi(F)-\varepsilon)n.$$
\end{lem}
\begin{proof}
Suppose, to the contrary, that $\delta(G)\leq(\pi(F)-\varepsilon)n$.
Denote $q=q(G)$ and $\delta=\delta(G)$ for brevity.
 Since
$$q^{2}-2q\delta+n\delta=(q-\delta)^{2}+\delta(n-\delta)>0,$$
by Lemma \ref{min},
we obtain
$$x_{\textup{min}}^{2}\leq \frac{\delta}{q^{2}-2q\delta+n\delta}.$$
Note that the right-hand side is increasing in $\delta$ and decreasing in $q$ on $[\delta,+\infty)$. Set $\varepsilon'=\pi(F)\varepsilon/(\pi(F)+\varepsilon)$.
By Theorem \ref{ess}, for sufficiently large $n$ we have
$$q\geq (2\pi(F)-\varepsilon')n>(\pi(F)-\varepsilon)n\geq \delta.$$
Combining these with $0<\varepsilon<1$ and $\pi(F)\geq1/2$, we derive
\begin{displaymath}
\begin{split}
x_{\textup{min}}^{2}n &\leq \frac{\delta n}{q^{2}-2q\delta+n\delta}
=\frac{\delta n}{n\delta+(q-\delta)^{2}-\delta^{2}}\\
&\leq \frac{\pi(F)-\varepsilon}{\pi(F)-\varepsilon
+(\pi(F)+\varepsilon-\varepsilon')^{2}-(\pi(F)-\varepsilon)^{2}}\\
&\leq \frac{\pi(F)-\varepsilon}{\pi(F)-\varepsilon
+4\pi(F)\varepsilon-2(\pi(F)+\varepsilon)\varepsilon'}\\
&= \frac{\pi(F)-\varepsilon}{\pi(F)-\varepsilon
+2\pi(F)\varepsilon}\\
&= 1- \frac{2\pi(F)\varepsilon}{\pi(F)-\varepsilon
+2\pi(F)\varepsilon}\\
&\leq  1-\varepsilon,
\end{split}
\end{displaymath}
which contradicts the assumption $x_{\textup{min}}^{2}>\frac{1-\varepsilon}{n}$. This completes the proof.
\end{proof}

The following signless Laplacian spectral stability theorem plays a crucial role in the spectral Tur\'an-type problems.

\begin{theorem}[\cite{ZLF2026}]\label{sst}
Let $r\ge 3$ and $F$ be a graph with $\chi(F)=r+1$. For every $\epsilon>0$, there exist $\delta>0$ and $n_{0}$ such that
if $G$ is an $F$-free graph on $n\geq n_{0}$ vertices with $q(G)\geq2(1- \frac{1}{r}-\delta)n$,  then $G$ can be obtained from $T_{n,r}$ by adding and deleting at most $\epsilon n^2$ edges.
\end{theorem}

The following two results on the signless Laplacian spectral radius are taken from \cite{ZYZ2026}.

\begin{lem}[\cite{ZYZ2026}]\label{t-color}
Let $r\geq3$, $t\geq1$ be fixed integers, $n_{1},\ldots, n_{r}$ be positive integers and $\sum_{i=1}^{r}n_{i}=n-t+1$. Then
$$q(K_{t-1}\vee K_{r}(n_{1},\ldots, n_{r}))\leq q(K_{t-1}\vee T_{n-t+1,r}),$$ with equality
 if and only if $K_{t-1}\vee K_{r}(n_{1},\ldots, n_{r})\cong K_{t-1}\vee T_{n-t+1,r}$.
\end{lem}

\begin{lem}[\cite{ZYZ2026}]\label{Ktv}
Let $r\geq3$, $t\geq1$ be fixed integers, and let $\alpha=\frac{2(r-1)}{r}$ and
$\beta=\frac{4(r-1)}{r(r-2)}$. Then for sufficiently large $n$,
$q(K_{t-1}\vee T_{n-t+1,r})=\alpha n+(t-1)\beta+O(\frac{1}{n})$.
\end{lem}

\begin{lem}\label{lem:2.2}
Let $G$ be a graph and  $\mathbf{x}$ be a positive unit eigenvector corresponding to $q(G)$.
Let $u,v\in V(G)$ with $uv\notin E(G)$ and $N_G(u)\neq N_G(v)$, and let $G'$ be the graph obtained from $G$ by deleting all
edges  incident with  $v$ and then adding all edges between $v$ and $N_{G}(u)$. If
\begin{equation}\label{xz}
\sum_{w\in N_{G}(u)}(x_{w}+x_{u})^{2}-q(G) x_{u}^{2}
\ge
\sum_{w\in N_{G}(v)}(x_{w}+x_{v})^{2}-q(G) x_{v}^{2},
\end{equation}
then $q(G')>q(G)$.
\end{lem}

\begin{proof}
Define a vector $\mathbf{y}$ on $V(G')$ by
$$
y_z=
\begin{cases}
x_{u},& z=v,\\
x_z,& z\neq v.
\end{cases}
$$
Then $\|\mathbf{y}\|^2=1-x_{v}^2+x_{u}^2>0$,
and so
\begin{align*}
\mathbf{y}^TQ(G')\mathbf{y}
&=\sum_{ij\in E(G')}(y_i+y_j)^2\\
&=q(G)-\sum_{w\in N_G(v)}(x_w+x_{v})^2
+\sum_{w\in N_G(u)}(x_w+x_{u})^2\\
&\geq q(G)+q(G)(x_{u}^2-x_{v}^2)\\
&=q(G)\|\mathbf{y}\|^2.
\end{align*}
Hence
\begin{equation}\label{xz-2}
q(G')\geq \frac{\mathbf{y}^T Q(G')\mathbf{y}}{\|\mathbf{y}\|^2}\geq q(G).
\end{equation}

If (\ref{xz}) is strict, the second inequality in (\ref{xz-2}) is strict, and therefore $q(G')>q(G)$.

Assume now that \eqref{xz} holds with equality and  $q(G')=q(G)$. Then the Rayleigh
quotient of $\mathbf{y}$ attains $q(G')$, implying that $\mathbf{y}$ is a positive eigenvector of $G'$ for
$q(G')$.
Since $N_G(u)\neq N_G(v)$, the symmetric difference $N_G(u)\triangle N_G(v)$ is nonempty.

Choose $z\in N_G(u)\triangle N_G(v)$. If $z\in N_G(u)\backslash N_G(v)$, then the edge $vz$ is
added in $G'$. The eigenequations at $z$ for $G$ and $G'$ give
\[
q(G)x_z=d_G(z)x_z+\sum_{i\in N_G(z)}x_i,
\]
and
\[
q(G'
)x_z=(d_G(z)+1
)x_z+x_u+\sum_{i\in N_G(z)}x_i.
\]
Since $q(G')=q(G)$, subtracting the above two equations yields  $0=x_z+x_u$, contradicting $x_z,x_u>0$.

If instead $z\in N_G(v)\backslash N_G(u)$, the edge $vz$ is deleted in $G'$. The eigenequation for $G'$ at $z$ gives
\[
q(G'
)x_z=(d_G(z)-1
)x_z+\sum_{i\in N_G(z)\setminus\{v\}}x_i.
\]
Comparing with the eigenequation for $G$ at $z$ and using $q(G')=q(G)$ yields $0=x_z+x_v$, again impossible.
Thus $q(G')>q(G)$ in all cases.
\end{proof}

Firstly, we demonstrate that the conclusion of Theorem \ref{thm:main1} holds under a stronger condition.

\begin{lem}\label{pro-lem}
For integers $1\le t\le \frac{s}{2}$ and $r\ge 3$, let
\[
F=\bigl(tK_{2}\cup \overline{K}_{s-2t}\bigr)\vee T_{s(r-1),r-1}.
\]
For sufficiently large $n$, if $G \in \mathrm{EX}_{q}(n, F)$ and
$\mathbf{x} = (x_1, \ldots, x_n)$ is a nonnegative unit eigenvector corresponding to $q(G)$
satisfying $x_{\min}^2 \ge \frac{1}{n}\bigl(1 - \frac{1}{\log\log n}\bigr)$, then
$G \cong K_{t-1} \vee T_{n-t+1,\, r}$.
\end{lem}

\section{Proof of Lemma \ref{pro-lem}}
In this section, we prove Lemma \ref{pro-lem}.
Let $F=\bigl(tK_{2}\cup \overline{K}_{s-2t}\bigr)\vee T_{s(r-1),r-1}$, $m = |V(F)|$ and $G\in \mathrm{EX}_{q}(n,F)$.
 We now proceed under the assumption of Lemma~\ref{pro-lem}, which ensures that $x_{\min}^2 \ge \frac{1}{n}\bigl(1 - \frac{1}{\log\log n}\bigr)$. Fix a  sufficiently small constant $\varepsilon$. Then $x_{\min}^2 \ge \frac{1}{n}(1 - \varepsilon)$  for sufficiently large $n$. By   Lemma
~\ref{mindu} and $\chi(F)=r+1$ , we obtain $$\delta(G)>\Big(\frac{r-1}{r}-\varepsilon\Big)n.$$

\begin{lem}\label{partition}
There exists a partition $V(G)=V_{1}\cup\cdots\cup V_{r}$ such that $\sum_{1\leq i<j\leq r}e(V_{i},V_{j})$ is maximum, $\sum_{i=1}^{r}e(G[V_{i}])\leq\frac{\varepsilon^{2}n^{2}}{4}$ and $\left||V_{i}|-\frac{n}{r}\right|\leq\varepsilon n$ for all $i\in[r]$.
\end{lem}

\begin{proof}
Since $\chi(F)=r+1,$ $K_{t-1}\vee T_{n-t+1,r}$ is $F$-free and $n$ is sufficiently large,  Lemma~\ref{Ktv} gives
$$q(G)\geq q(K_{t-1}\vee T_{n-t+1,r})>2\Big(1-\frac{1}{r}\Big)n-\frac{C}{n},$$ where $C>0$ is a constant.
By Theorem~\ref{sst} (letting $\epsilon=\varepsilon^{2}/4$), we have $e(G)\geq e(T_{n, r})-\varepsilon^{2} n^{2}/4$. Furthermore, there exists a vertex partition $V(G)=U_{1}\cup\cdots\cup U_{r}$ with $\lfloor\frac{n}{r}\rfloor\leq|U_{i}|\leq\lceil\frac{n}{r}\rceil$ such that
$\sum_{i=1}^{r} e(G[U_{i}])\leq\varepsilon^{2} n^{2}/4$. Choose a partition $V(G)=V_{1}\cup\cdots\cup V_{r}$ such that
$\sum_{1\leq i<j\leq r} e\left(V_{i}, V_{j}\right)$ attains the maximum. Then
\[
\sum_{i=1}^{r} e\left(G[V_{i}]\right)\leq\sum_{i=1}^{r} e\left(G[U_{i}]\right)\leq \frac{\varepsilon^{2}n^{2}}{4}.
\]

Let $\max_{1\leq i\leq r}\left|\left|V_{i}\right|-\frac{n}{r}\right|=d$. Without loss of generality assume that $\left||V_{1}|-\frac{n}{r}\right|=d$. Then we have
\begin{align*}
e(G)&=\sum_{1\leq i<j\leq r} e(V_{i}, V_{j})+\sum_{i=1}^{r} e(G[V_{i}])\\
&\leq\sum_{1\leq i<j\leq r}\left|V_{i}\right|\left|V_{j}\right|+\frac{\varepsilon^{2}n^{2}}{4}\\
&=\left|V_{1}\right|\left(n-\left|V_{1}\right|\right)+\sum_{2\leq i<j\leq r}\left|V_{i}\right|\left|V_{j}\right|
+\frac{\varepsilon^{2}n^{2}}{4}\\
&=\left|V_{1}\right|\left(n-\left|V_{1}\right|\right)+\frac{1}{2}\left(\left(\sum_{i=2}^{r}\left|V_{i}\right|\right)^{2}
-\sum_{i=2}^{r}\left|V_{i}\right|^{2}\right)+\frac{\varepsilon^{2}n^{2}}{4}.
\end{align*}
By the Cauchy-Schwarz inequality, we have $\left(\sum_{i=2}^{r}\left|V_{i}\right|\right)^{2}\leq(r-1)\sum_{i=2}^{r}\left|V_{i}\right|^{2}$. Together with $\left||V_{1}|-\frac{n}{r}\right|=d$, we get
\begin{align*}
e(G)&\leq\left|V_{1}\right|\left(n-\left|V_{1}\right|\right)+\frac{r-2}{2(r-1)}\left(n-\left|V_{1}\right|\right)^{2}
+\frac{\varepsilon^{2}n^{2}}{4}\\
&=\frac{r-1}{2 r} n^{2}-\frac{r}{2 r-2} d^{2}+\frac{\varepsilon^{2}n^{2}}{4}.
\end{align*}

On the other hand, since $e(T_{n, r})\geq \frac{r-1}{2 r} n^2-\frac{r}{8}$, we have
\[
e(G)\geq e\left(T_{n, r}\right)-\frac{\varepsilon^{2}n^{2}}{4}\geq\frac{r-1}{2 r} n^2-\frac{r}{8}-\frac{\varepsilon^{2}n^{2}}{4}.
\]
Together with the above inequality, we have
\[
d\leq\sqrt{\frac{(r-1)\varepsilon^{2} n^{2}}{r}+\frac{r-1}{4}}<\varepsilon n,
\]
as desired.
\end{proof}

Let $W=\cup_{i=1}^{r}W_{i}$, where $W_{i}=\{v\in V_{i}: d_{V_{i}}(v)\geq\varepsilon n\}$, and let $\overline{V}_{i}=V_{i}\backslash W_{i}$ for $i\in[r]$.

\begin{lem}\label{lem:4.1}
$|W|\le M$, where
\[
M:=\left\lceil 4^{(r-1)m}\left(\frac{2}{\varepsilon}\right)^m rm\right\rceil
\]
is  a constant independent of $n$.
\end{lem}
\begin{proof}
We first show that $|W| \le \frac{\varepsilon n}{2}$.  By the definition of $W_i$, we have $d_{V_i}(v)\geq\varepsilon n$ for every $v\in W_i$. Combining with Lemma~\ref{partition},
\[
\varepsilon n|W|=\sum_{i=1}^r \varepsilon n|W_i| \le \sum_{i=1}^r \sum_{v\in W_i} d_{V_i}(v)
\le \sum_{i=1}^r \sum_{v\in V_i} d_{V_i}(v)
=2\sum_{i=1}^r e(G[V_i]) {\le} \frac{\varepsilon^2}{2}n^2,
\]
which implies $|W|\le \varepsilon n/2$.

Fix $1\le i\le r$. For each $j\in [r]\backslash \{i\}$, we have
$d_{V_i}(v)\le d_{V_j}(v)$ for any $v\in V_i$ as $\sum_{1\leq i<j\leq r}e(V_{i},V_{j})$ is maximum.
Therefore,
$d_{V_j}(v)\ge \frac12\bigl(d_{V_i}(v)+d_{V_j}(v)\bigr)$ for any $v\in V_i$.
By Lemma \ref{partition}, $|V_s|\le (\frac1r+\varepsilon)n$ for any $1\le s\le r$.

If $v\in W_i$, then for any $j\in[r]\backslash\{i\}$,
\[
d_{V_j}(v)\ge \frac12\left(d_G(v)-\sum_{s\in [r]\setminus\{i,j\}}|V_s|\right)
\ge \frac12\left(\frac{r-1}{r}n-\varepsilon n-(r-2)\left(\frac1r+\varepsilon\right)n\right)
= \left(\frac{1}{2r}-\frac{r-1}{2}\varepsilon\right)n.
\]
It follows that
\begin{eqnarray}\label{lower bound of W}
d_{\overline{V}_j}(v)\ge d_{V_j}(v)-|W|
\ge \left(\frac{1}{2r}-\frac{r-1}{2}\varepsilon\right)n-\frac{\varepsilon n}{2}
=\frac{n}{2r}-\frac12 r\varepsilon n.
\end{eqnarray}

If $v\in \overline{V}_i$, then $d_{V_i}(v)< \varepsilon n$. Thus for any $j\in [r]\backslash\{i\}$,
\begin{eqnarray*}
d_{V_j}(v)&\ge& d_G(v)-d_{V_i}(v)-\sum_{s\in [r]\setminus\{i,j\}}|V_s|\\
&\ge & \left(\frac{r-1}{r}-\varepsilon\right)n-\varepsilon n-(r-2)\left(\frac1r+\varepsilon\right)n
= \left(\frac1r-r\varepsilon\right)n.
\end{eqnarray*}
Hence
\begin{eqnarray}\label{lower bound of V}
d_{\overline{V}_j}(v)\ge d_{V_j}(v)-|W|
\ge \left(\frac1r-r\varepsilon\right)n-\frac{\varepsilon n}{2}
\ge \frac{n}{r}-(r+1)\varepsilon n.
\end{eqnarray}

In the following, we show that $|W_1|\le 4^{(r-1)m}(2/\varepsilon)^m m$. Let $m\le a\le b$ be reals and $\binom{x}{m}=x(x-1)\cdots(x-m+1)/m!$. Since $\frac{a-t}{b-t}\ge \frac{a-m}{b}$ for $0\le t\le m-1$, we have 
\begin{eqnarray}\label{binomial inequality}
\frac{\binom{a}{m}}{\binom{b}{m}}
=\prod_{t=0}^{m-1}\frac{a-t}{b-t}\ge \left(\frac{a-m}{b}\right)^m.
\end{eqnarray}

Next we construct, by induction on $k$ with $1\le k\le r$, sets $Z_k\subseteq \overline{V}_k$ with $|Z_k|=m$, together with sets $ W_1\supseteq W_{1,1}\supseteq W_{1,2}\supseteq\cdots\supseteq W_{1,k}, $ such that the following properties hold for every $k\in\{1,\dots,r\}$:
\begin{itemize}
\item[(a)] every vertex of $W_{1,k}$ is completely joined to $Z_1\cup\cdots\cup Z_k$;

\item[(b)] $Z_1,\dots,Z_k$ are pairwise completely joined, i.e., their union contains a copy of $K_k(m,\dots,m)$;

\item[(c)] $|W_{1,k}|\ge \left(\frac14\right)^{(k-1)m}\left(\frac{\varepsilon}{2}\right)^m |W_1|$.

\end{itemize}


\noindent\textit{Base case $k=1$.} For any $v\in W_1$, we have
$d_{\overline V_1}(v)\ge d_{V_1}(v)-|W_1|\ge \varepsilon n-\frac{1}{2}\varepsilon n=\frac{1}{2}\varepsilon n$.
Let
$
Y_1=\{(w,Z)\mid w\in W_1,\ Z\subseteq N_{\overline V_1}(w),\ |Z|=m\}.
$
By double counting, we have
\[
|W_1|\binom{\varepsilon n/2}{m}\le |Y_{1}| = \sum_{\substack{Z\subseteq \overline{V}_1\\|Z|=m}} \big|\{w\in W_1: Z\subseteq N(w)\}\big|.
\]
Since $|\overline{V}_1|\le (\frac1r+\varepsilon)n$, some $Z_1\subseteq \overline{V}_1$ of size $m$ has at least
\[
|W_1|\frac{\binom{\varepsilon n/2}{m}}{\binom{(1/r+\varepsilon)n}{m}}\ge \left(\frac{\varepsilon}{2}\right)^m |W_1|
\]
common neighbours in $W_1$. Indeed, by inequality (\ref{binomial inequality}) (applied with $a=\varepsilon n/2\ge m$ and $b=(1/r+\varepsilon)n$),
\[
\frac{\binom{\varepsilon n/2}{m}}{\binom{(1/r+\varepsilon)n}{m}}
\ge \left(\frac{\varepsilon n/2-m}{(1/r+\varepsilon)n}\right)^m
=\left(\frac{\varepsilon/2-m/n}{1/r+\varepsilon}\right)^m
\ge \left(\frac{\varepsilon}{2}\right)^m,
\]
where the last inequality   holds as $n$ is sufficiently large. 
Take $W_{1,1}$ to be this set of common
neighbours. So (a)--(c) hold for $k=1$.

\medskip
\noindent\textit{Induction step $k\to k+1$ ($k<r$).} Let $V_{k+1}'\subseteq \overline{V}_{k+1}$ be the set of common neighbours of the
$km$ vertices of $Z_1\cup\cdots\cup Z_k$. Since
$ |\overline V_{k+1}|\le (\frac{1}{r}+\varepsilon)n$
and $d_{\overline V_{k+1}}(u)\ge \frac{n}{r}-(r+1)\varepsilon n$ for any
$u\in \cup_{1\le i\le k}Z_i$, we have
$d_{\overline V_{k+1}}(u)\ge |\overline V_{k+1}|-(r+2)\varepsilon n$ for such $u$.
By Lemma \ref{dx},
\begin{align*}
|V_{k+1}'|
&=\Bigl|\bigcap_{u\in \cup_{1\le i\le k}Z_i} N_{\overline V_{k+1}}(u)\Bigr|
\ge km(|\overline V_{k+1}|-(r+2)\varepsilon n)-(km-1)|\overline V_{k+1}|\\
&\ge |\overline V_{k+1}|-km(r+2)\varepsilon n\ge |\overline V_{k+1}|-rm(r+2)\varepsilon n.
\end{align*}

For any $w\in W_1$, by (\ref{lower bound of W}), we have $d_{\overline{V}_{k+1}}(w)\ge \frac{n}{2r}-\frac12 r\varepsilon n$, and so
\[
d_{V_{k+1}'}(w)\ge d_{\overline{V}_{k+1}}(w)-\big|\overline{V}_{k+1}\setminus V_{k+1}'\big|
\ge \frac{n}{2r}-rm(r+3)\varepsilon n.
\]
Let
$
Y_{k+1}=\{(w,Z)\mid w\in W_{1,k},\ Z\subseteq N_{V_{k+1}'}(w),\ |Z|=m\}.
$
Exactly as in the base case, by double counting,
there is a subset $Z_{k+1}\subseteq V_{k+1}'$ of size $m$ whose vertices have at
least
$$\frac{|W_{1,k}|\binom{\frac{n}{3r}}{m}}{\binom{(\frac{1}{r}+\varepsilon)n}{m}}
\ge \Bigl(\frac{1}{4}\Bigr)^{m}|W_{1,k}|$$
common neighbors in $W_{1,k}$. Let $W_{1,k+1}\subseteq W_{1,k}$ be this set of common
neighbors. Then
\[
|W_{1,k+1}|\ge \Bigl(\frac{1}{4}\Bigr)^{m}|W_{1,k}|
\ge \Bigl(\frac{1}{4}\Bigr)^{km}\Bigl(\frac{\varepsilon}{2}\Bigr)^{m}|W_1|.
\]
Moreover, $Z_{k+1}\subseteq V_{k+1}'$ is joined to all of $Z_1,\dots,Z_k$ by the definition of $V_{k+1}'$, while $W_{1,k+1}\subseteq
W_{1,k}$ inherits (a). Thus (a)--(c) hold at step $k+1$, completing the induction.

 By (c) with $k=r$,
\[
|W_{1,r}|\ge \left(\frac14\right)^{(r-1)m}\left(\frac{\varepsilon}{2}\right)^m |W_1|.
\]
So, if $|W_1|\ge 4^{(r-1)m}(2/\varepsilon)^m m$, then $|W_{1,r}|\ge m$. It follows that the subgraph induced by $W_{1,r},Z_1,\dots,Z_r$ contains a complete $(r+1)$-partite graph each part of size at least $m$, contradicting that $G$ is $F$-free. Thus $|W_1|<4^{(r-1)m}(2/\varepsilon)^m m$. Similarly $|W_i|<4^{(r-1)m}(2/\varepsilon)^m m$ for $2\le i\le r$. Hence
\[
|W|<4^{(r-1)m}\left(\frac{2}{\varepsilon}\right)^m rm \leq  M,
\]
completing the proof.
\end{proof}

\begin{lem}\label{lem:4.2}
Let $j\in[r]$ and let $1\leq \ell\leq m$ .  If
$u_{1},u_{2},\ldots,u_{\ell}\in \cup_{i\in [r]\backslash\{j\}}\overline V_{i}$,
then there are at least $\frac{n}{2r}$ vertices in $\overline V_j$ which are
adjacent to all the vertices $u_{1},u_{2},\ldots,u_{\ell}$ in $G$.
\end{lem}

\begin{proof}
By (\ref{lower bound of V}), for any $i\in[\ell]$, $d_{\overline{V}_j}(u_{i}) \geq \frac{n}{r} - (r+1)\varepsilon n$.
From Lemmas \ref{dx} and \ref{partition}, it follows that
\begin{align*}
\Bigl|\bigcap_{i\in [\ell]} N_{\overline V_{j}}(u_{i})\Bigr|
&\ge \ell\Bigl(\frac{n}{r} - (r+1)\varepsilon n\Bigr)-(\ell-1)|\overline V_{j}|\\
&\ge \ell\Bigl(\frac{n}{r} - (r+1)\varepsilon n\Bigr)-(\ell-1)\Bigl(\frac{1}{r}+\varepsilon\Bigr)n\\
&=\frac{n}{r} - (\ell r+2\ell-1)\varepsilon n\\
&\ge \frac{n}{2r},
\end{align*}
where the last inequality uses $\ell\le m$ and the fact that $\varepsilon$
was chosen sufficiently small.
\end{proof}

Let $x_{\textup{max}}=\max\{x_{1},\ldots,x_{n}\}$.  Assume that $u^{*}$ is a vertex with $x_{u^{*}}=x_{\textup{max}}$. Let $v^{*}$ be a vertex in $V(G)\backslash W$ such that $x_{v^{*}}=\max_{v\in V(G)\backslash W}x_{v}$.

\begin{lem}\label{lem:4.3}
 $x_{\textup{max}}\leq \frac{16}{\sqrt{n}}$.
\end{lem}

\begin{proof}
Recall that $|W|\leq M$ and $q(G)\geq 2(1-1/r)n-\frac{C}{n}$. Since $(q(G)-d_{G}(u^{*}))x_{u^{*}}\leq|W|x_{u^{*}}+(n-|W|)x_{v^{*}}$, we have
\begin{align*}
x_{v^{*}}&\geq\frac{(q(G)-d_{G}(u^{*})-|W|)x_{u^{*}}}{n}\\
&\geq \frac{(q(G)-n-M)x_{u^{*}}}{n}\\
&\geq (1-2/r-\varepsilon)x_{u^{*}}.
\end{align*}

It follows from $r\geq3$ that $x_{v^{*}}>\frac{1}{4}x_{u^{*}}$.
Assume that $v^{*}\in V_{i_{0}}$, where $1\leq i_{0}\leq r$. Then $|N_{\overline{V}_{i_{0}}}(v^{*})|\leq|N_{V_{i_{0}}}(v^{*})|<\varepsilon n$ as $v^{*}\notin W$. Hence
\begin{align*}
(q(G)-d_{G}(v^{*}))x_{v^{*}}
&= \biggl(\sum_{v\in N_{W}(v^{*})}x_{v}\biggr)
   + \biggl(\sum_{v\in N_{\overline{V}_{i_{0}}}(v^{*})}x_{v}\biggr)
   + \biggl(\sum_{v\in \cup_{i\in[r]\backslash \{i_{0}\}}N_{\overline{V}_{i}}(v^{*})}x_{v}\biggr)\\
&\leq |W|x_{u^{*}} + \varepsilon n x_{v^{*}} + \sum_{v\in\cup_{i\in[r]\backslash \{i_{0}\}}\overline{V}_{i}}x_{v}\\
&\leq (4M+\varepsilon n)x_{v^{*}} + \sum_{v\in\cup_{i\in[r]\backslash \{i_{0}\}}\overline{V}_{i}}x_{v}.
\end{align*}
It follows that
\[
\sum_{v\in\cup_{i\in[r]\backslash \{i_{0}\}}\overline{V}_{i}}x_{v}
\geq (q(G)-d_{G}(v^{*})- 4M-\varepsilon n)x_{v^{*}}
\geq \Bigl(1-\frac{2}{r}-2\varepsilon\Bigr)nx_{v^{*}}
\geq \frac{nx_{u^{*}}}{16}.
\]
Hence
$$\frac{nx_{u^{*}}}{16}\leq \bigg(\sum_{v\in\cup_{i\in[r]\backslash \{i_{0}\}}\overline{V}_{i}}x_{v}^{2}\bigg)^{\frac{1}{2}}
(|V(G)|)^{\frac{1}{2}}\leq \sqrt{n},$$
which implies $x_{u^{*}}\leq \frac{16}{\sqrt{n}}$. This completes the proof.
\end{proof}

For   $1\leq i\leq r,$ let $R_{i}$ be the set of vertices in $\overline{V}_{i}$ that are the end
vertices of  a maximum matching  of $G[\overline{V}_{i}]$.
For any  $1\leq i\leq r$ and $u\in \overline{V}_{i}\backslash R_{i}$,
 define
\[
\alpha(u)=\{uv\in E(G)\mid  v\in (V_{i}\backslash\{u\})\cup W\},
\]
and
\[
\beta(u)=\{uv\notin E(G)\mid  v\in \cup_{ j\in[r]\backslash\{i\}}\overline{V}_{j}\}.
\]


\begin{lem}\label{lem:4.4}
For any  $1\leq i\leq r$ and $u\in \overline{V}_{i}\backslash R_{i}$, we have $|\alpha(u)|\leq M+2t$
and $|\beta(u)|\leq 512(M+2t)$.
\end{lem}

\begin{proof}
For convenience, assume $i=1$. We first claim that $\mu(G[\overline{V}_{1}])\le t-1$.
Otherwise, let $\{u_{1,1}u_{1,2},\ldots,u_{1,2t-1}u_{1,2t}\}$ be a $t$-matching in
$G[\overline{V}_{1}]$, and let $u_{1,2t+1},\ldots,u_{1,s}\in\overline{V}_{1}$.
By Lemma~\ref{lem:4.2}, there exist $s$ vertices
$u_{2,1},\ldots,u_{2,s}\in\overline{V}_{2}$ that are adjacent to all vertices
$u_{1,1},\ldots,u_{1,s}$.
Applying Lemma~\ref{lem:4.2} again, there exist $s$ vertices
$u_{3,1},\ldots,u_{3,s}\in\overline{V}_{3}$ that are adjacent to all vertices
$u_{1,1},\ldots,u_{1,s},u_{2,1},\ldots,u_{2,s}$.
Repeating this process, we eventually obtain a subgraph of $G$ that contains
$F$ as a subgraph, a contradiction.
Hence $\mu(G[\overline{V}_{1}])\le t-1$, which implies that
 $|R_{1}|\leq 2(t-1)$.
Thus, $\overline{V}_{1}\backslash R_{1}$ is  an independent set.
This implies that $|\alpha(u)|\leq M+2t$ for any $u\in \overline{V}_{1}\backslash R_{1}$.

Suppose to the contrary that $|\beta(u)|>512(M+2t)$.
Let $G'$ be the graph obtained from $G$ by deleting all edges in $\alpha(u)$ and
adding all non-edges in $\beta(u)$.
Then, by Lemma~\ref{lem:4.3} and the bound
 $x_{\min}^2 \ge \frac{1}{n}\bigl(1 - \frac{1}{\log\log n}\bigr)>\frac{1}{2n}$, we have
\[
\sum_{uv\in\beta(u)}(x_{u}+x_{v})^{2}-\sum_{uv\in\alpha(u)}(x_{u}+x_{v})^{2}
\geq \frac{2}{n}|\beta(u)|- \frac{16^{2}\times4}{n}|\alpha(u)|>0.
\]
It follows that
\[
q(G')-q(G)
\geq \mathbf{x}^{T}(Q(G')-Q(G))\mathbf{x}
=\sum_{uv\in\beta(u)}(x_{u}+x_{v})^{2}-\sum_{uv\in\alpha(u)}(x_{u}+x_{v})^{2}>0,
\]
implying that $q(G')>q(G)$.
Since $G\in\mathrm{EX}_q(n,F)$, the graph $G'$ is not
$F$-free. Let $H\subseteq G'$ be a copy of $F$. Since $G$ and $G'$ differ
only in edges incident with $u$, while $G$ is $F$-free, we must have
$u\in V(H)$. Let $u_1,\ldots,u_{\ell}$ be the neighbours of $u$ in $H$.
By the construction of $G'$, all these vertices lie in
$\cup_{2\le i\le r}\overline V_i$. By Lemma \ref{lem:4.2}, at least $n/(2r)$
vertices of $\overline V_1$ are adjacent in $G$ to all of
$u_1,\ldots,u_{\ell}$. For sufficiently large $n$, we have $n/(2r)>m$, so among these common neighbours we
may choose $u'\in\overline V_1\backslash V(H)$. Replacing $u$ by $u'$ in
$H$ yields a copy of $F$ in $G$, a contradiction.
Hence, $|\beta(u)|\leq 512(M+2t)$.
\end{proof}

The ideas of the following two lemmas are directly derived from \cite{Zhang2026}.

\begin{lem}\label{lem:4.5}
Let $1\le i\le r$. Assume that $u_{1},u_{2},\ldots,u_{\tau}$ are symmetric vertices
of $G[V_{i}\cup W]$ and $\{u_{1},u_{2},\ldots,u_{\tau}\}\subseteq\overline{V}_{i}$.
Let $B$ denote the set of vertices in $V(G)\backslash(V_{i}\cup W)$ that are adjacent
to none of $u_{1},u_{2},\ldots,u_{\tau}$.
If $u_{1},u_{2},\ldots,u_{\tau}$ are not symmetric in $G$, then there exists a vertex
$v\notin B\cup(V_{i}\cup W)$ that is non-adjacent to at least
$\frac{1}{2m}\tau$ vertices among $u_{1},u_{2},\ldots,u_{\tau}$.
\end{lem}

\begin{proof}
We may assume $\frac{1}{2m}\tau\ge 1$; otherwise there is nothing to prove.
Then $\tau\ge 2m$. For convenience, suppose $i=1$.
Since $u_{1},u_{2},\ldots,u_{\tau}$ are symmetric in $G[V_{1}\cup W]$ and no vertex
of $B$ is adjacent to any  of $u_1,\ldots,u_{\tau}$, the set $\{u_{1},\dots,u_{\tau}\}$ remains
symmetric in $G[B\cup V_{1}\cup W]$.

Set $Z=V(G)\backslash\bigl(B\cup V_{1}\cup W\bigr)$.
If every $u_{i}$ is adjacent to all vertices of $Z$, then
$u_{1},\dots,u_{\tau}$ are symmetric in $G$, which contradicts the hypothesis.
Thus we may assume that $u_{1}$ is not adjacent to some $v_{1}\in Z$.
Construct $G'$ from $G$ by adding the edge $u_{1}v_{1}$.
Since $x_{\min}^{2}\ge\frac{1}{n}\bigl(1-\frac{1}{\log\log n}\bigr)>0$,
we have
\[
q(G')-q(G)
\geq\mathbf{x}^{T}\!\bigl(Q(G')-Q(G)\bigr)\mathbf{x}
=(x_{u_{1}}+x_{v_{1}})^{2}>0,
\]
so $q(G')>q(G)$. Because $G\in\mathrm{EX}_{q}(n,F)$, it follows that
$G'$ contains a subgraph  $F$. Clearly, $F$ contains the edge
$u_{1}v_{1}$.

Let $v_{1},\dots,v_{j}$ be the vertices of $Z$ that lie in $F$, where
$1\le j\le m$. We show that at most $m$ vertices among
$\{u_{1},\dots,u_{\tau}\}$ are adjacent to all of $v_{1},\dots,v_{j}$.
Otherwise, we can choose $u_{k}\notin V(F)$ with $1\le k\le\tau$ such that
$u_{k}$ is adjacent to every $v_{\ell}$ ($1\le\ell\le j$).
By the symmetry of $\{u_{1},\dots,u_{\tau}\}$ in $G[B\cup V_{1}\cup W]$,
every neighbour of $u_{1}$ in $F$ is also adjacent to $u_{k}$ in $G$.
Let $F'$ be the subgraph of $G$ induced by
$(V(F)\setminus\{u_{1}\})\cup\{u_{k}\}$.
Then $F\subseteq F'$, which contradicts the $F$-freeness of $G$.

Consequently, at least $\tau-m$ vertices among $\{u_{1},\dots,u_{\tau}\}$ fail
to be adjacent to some vertex in $\{v_{1},\dots,v_{j}\}$. By the pigeonhole
principle, there exists a vertex $v\in\{v_{1},\dots,v_{j}\}$ that is non-adjacent
to at least
\[
\frac{\tau-m}{j}
\ge\frac{\tau-m}{m}
\ge\frac{1}{2m}\tau
\]
vertices among $u_{1},u_{2},\ldots,u_{\tau}$.
This completes the proof.
\end{proof}

\begin{lem}\label{lem:4.6}
Let $1\le i\le r$. Assume that $u_{1},u_{2},\ldots,u_{\tau}$ are symmetric vertices
of $G[\overline{V}_{i}\backslash R_{i}]$. Then there exist at least
$\bigl(\tfrac{1}{2m}\bigr)^{512(M+2t)}2^{-(M+2t)}\tau$ vertices among
$u_{1},u_{2},\ldots,u_{\tau}$ that are symmetric in $G$.
\end{lem}

\begin{proof}
We may assume
$\bigl(\tfrac{1}{2m}\bigr)^{512(M+2t)}2^{-(M+2t)}\tau\ge 1$;
otherwise there is nothing to prove. Then
$\tau\ge (2m)^{512(M+2t)}2^{M+2t}$.
For convenience, suppose $i=1$. By Lemma \ref{lem:4.4},
$|\beta(u_{j})|\le 512(M+2t)$ for each $1\le j\le\tau$.
Since  there are at most $2^{M+2t}$ possible
neighborhoods in $W\cup R_{1}$,
at least $\tau 2^{-(M+2t)}$ vertices among $u_{1},\dots,u_{\tau}$ are symmetric
in $G[V_{1}\cup W]$, i.e., have identical neighbourhoods in $W\cup R_{1}$.
Without loss of generality, assume that
$u_{1},\dots,u_{\tau_{0}}$ are such vertices, where $\tau_{0}\ge\tau 2^{-(M+2t)}$.

Let $B_{0}$ denote the set of vertices in $V(G)\backslash(V_{1}\cup W)$ that are
adjacent to none of $u_{1},\dots,u_{\tau_{0}}$. If $u_{1},\dots,u_{\tau_{0}}$ are
symmetric in $G$, the lemma follows from $\tau_{0}\ge\tau 2^{-(M+2t)}$. Hence we may
assume they are not symmetric in $G$. By Lemma \ref{lem:4.5}, there exists a
vertex $v_{1}\notin B_{0}\cup(V_{1}\cup W)$ that is non-adjacent to at least
$\frac{1}{2m}\tau_{0}$ vertices among $u_{1},\dots,u_{\tau_{0}}$. Without loss of
generality, let $u_{1},\dots,u_{\tau_{1}}$ be the vertices non-adjacent to
$v_{1}$, where $\tau_{1}\ge\frac{1}{2m}\tau_{0}\ge\frac{1}{2m}2^{-(M+2t)}\tau$.

Define $B_{1}$ as the set of vertices in $V(G)\backslash(V_{1}\cup W)$ adjacent
to none of $u_{1},\dots,u_{\tau_{1}}$. Then $v_{1}\in B_{1}$, so $|B_{1}|\ge1$.
If $u_{1},\dots,u_{\tau_{1}}$ are symmetric in $G$, the lemma again follows.
Otherwise, Lemma \ref{lem:4.5} yields a vertex
$v_{2}\notin B_{1}\cup(V_{1}\cup W)$ non-adjacent to at least
$\frac{1}{2m}\tau_{1}$ vertices among $u_{1},\dots,u_{\tau_{1}}$. Let
$u_{1},\dots,u_{\tau_{2}}$ be those vertices, where
$\tau_{2}\ge\frac{1}{2m}\tau_{1}\ge\bigl(\frac{1}{2m}\bigr)^{2}2^{-(M+2t)}\tau$.

Iterating this procedure, we obtain an increasing chain
$B_{0}\subsetneq B_{1}\subsetneq\cdots\subsetneq B_{\ell}$ and  $\tau_{\ell}$ vertices
$u_{1},\dots,u_{\tau_{\ell}}$ that are symmetric in
$G[B_{\ell}\cup(V_{1}\cup W)]$, satisfying
$\tau_{\ell}\ge\bigl(\tfrac{1}{2m}\bigr)^{\ell}2^{-(M+2t)}\tau$.
Moreover, no vertex of $B_\ell$ is adjacent to any of
$u_1,\ldots,u_{\tau_\ell}$. Since
\[
B_\ell\subseteq V(G)\backslash(V_1\cup W)
=\cup_{j\in[r]\backslash\{1\}}\overline V_j,
\]
and every vertex of $B_\ell$ is non-adjacent to $u_1$, we have
\[
\ell\le |B_\ell|\le |\beta(u_1)|\le 512(M+2t).
\]
Let $C=512(M+2t)$. If the process does not terminate at step $C$, then
Lemma 3.6 would produce $B_{C+1}$, giving
\[
C+1\le |B_{C+1}|\le |\beta(u_1)|\le C,
\]
a contradiction. Therefore  the process terminates at some $\ell\le C$,
 and at that point $u_1,\dots,u_{\tau_\ell}$ are pairwise  symmetric in $G$.
 Consequently,
\[
\tau_{\ell}
\ge\bigl(\tfrac{1}{2m}\bigr)^{\ell}2^{-(M+2t)}\tau
\ge\bigl(\tfrac{1}{2m}\bigr)^{512(M+2t)}2^{-(M+2t)}\tau,
\]
which completes the proof.
\end{proof}

We now prove Lemma \ref{pro-lem}.

\begin{proof}[{\bf Proof of Lemma \ref{pro-lem}}]

In the proof of Lemma~\ref{lem:4.4}, we have shown that
$|R_{i}|\le 2(t-1)$ for all $i\in[r]$.
For each $i\in[r]$,  define
$\overline{V}_{i}'=\overline{V}_{i}\backslash R_{i}$.
Then  $\overline{V}_{i}'$ is an independent set.

Fix $1\le i\le r$.
Let $u_{1},u_{2},\ldots,u_{\tau}$ be symmetric vertices in $G$ with $\tau$
maximal, where $u_{j}\in\overline{V}_{i}'$ for all $1\le j\le\tau$.
We first show that $\tau\ge m+1$.
For sufficiently large $n$, we have
$2^{-2t}|\overline{V}_{i}'|\ge (2m)^{512(M+2t)}2^{M+2t}(m+1)$.
Thus, at least $(2m)^{512(M+2t)}2^{M+2t}(m+1)$ vertices in $\overline{V}_{i}'$
are adjacent to $R_{i}$ in the same way; consequently, these vertices are
symmetric in $G[\overline{V}_{i}]$.
By Lemma~\ref{lem:4.6}, at least $m+1$ of them are symmetric in $G$.
Hence $\tau\ge m+1$.

We next show that
\begin{equation}\label{ast}
|\overline{V}_{i}'\backslash\{u_{1},\dots,u_{\tau}\}|
\le (2m)^{512(M+2t)}2^{M+4t}(m+1)-1.
\end{equation}
Otherwise,
 $\overline{V}_{i}'\backslash\{u_{1},\dots,u_{\tau}\}$ contains at least
$(2m)^{512(M+2t)}2^{M+4t}(m+1)$ vertices.
Since $|R_{i}|<2t$, at least
$(2m)^{512(M+2t)}2^{M+2t}(m+1)$ of them are adjacent to $R_{i}$ in the same way.
As $\overline{V}_{i}'$ is an independent set, these vertices are symmetric in
$G[\overline{V}_{i}]$.
By Lemma~\ref{lem:4.6}, there exist $m+1$ vertices among them,
say $v_{1},\dots,v_{m+1}$, that are symmetric in $G$.
By the maximality of $\tau$, the vertices $u_{1}$ and $v_{1}$ are not symmetric
in $G$, i.e., $N_{G}(u_{1})\neq N_{G}(v_{1})$.
Without loss of generality, assume
$$\sum_{w\in N_{G}(u_{1})}(x_{w}+x_{u_{1}})^{2}-q(G) x_{u_{1}}^{2}
\ge
\sum_{w\in N_{G}(v_{1})}(x_{w}+x_{v_{1}})^{2}-q(G) x_{v_{1}}^{2}.$$

Let $G'$ be the graph obtained from $G$ by deleting all edges incident with
$v_{1}$ and adding all edges between $v_{1}$ and $N_{G}(u_{1})$.
We claim that $G'$ is $F$-free. Suppose otherwise, and let $H$ be a
copy of $F$ in $G'$. If $v_1\notin V(H)$, then $H\subseteq G$, a
contradiction. Hence $v_1\in V(H)$. Since $\tau\ge m+1$, there is
some $u_k\in\{u_1,\ldots,u_\tau\}\backslash V(H)$. The vertices $u_k$
and $v_1$ are symmetric in $G'$, so replacing $v_1$ by $u_k$ in $H$
yields a copy of $F$ whose edges all belong to $G$, again a
contradiction. Hence $G'$ is $F$-free,  but Lemma~\ref{lem:2.2} gives $q(G')>q(G)$, contradicting the extremal choice of $G$.
Hence \eqref{ast} holds.

Set $\overline{V}_{i}''=\{u_{1},\dots,u_{\tau}\}$.
Since $|V_{i}\backslash\overline{V}_{i}'|\le M+2t$, we obtain
\begin{align}\label{m1}
|V_{i}\backslash\overline{V}_{i}''|
\le (2m)^{512(M+2t)}2^{M+4t}(m+1)+M+2t
=:M_{1}.
\end{align}

\begin{myCli}\label{1}
For any $i\in[r]$ and  $u\in\overline{V}_{i}''$,
$N_{G}(u)=V(G)\backslash\overline{V}_{i}$.
\end{myCli}
\begin{proof}
For each $j\in[r]$, the set $\overline{V}_{j}''$ consists of vertices that are
pairwise symmetric in $G$, and
$|V_{j}\backslash\overline{V}_{j}''|\le M_{1}$.
Fix $1\le i\le r$.
For any $v\in V(G)\backslash\overline{V}_{i}''$, either $v$ is adjacent to all
vertices of $\overline{V}_{i}''$ or to none of them.

By (\ref{lower bound of W}) and (\ref{lower bound of V}),
$d_{\overline{V}_{i}}(v)\ge\frac{n}{2r}-\frac{r\varepsilon n}{2}\ge\frac{n}{3r}$
for every $v\in V_{j}$ with $j\neq i$.
Hence every vertex of $V_{j}$ ($j\neq i$) is adjacent to all vertices of
$\overline{V}_{i}''$.
Moreover, $d_{V_{i}}(v)<\varepsilon n$ for $v\in\overline{V}_{i}$, so such
vertices are non-adjacent to $\overline{V}_{i}''$, whereas
$d_{V_{i}}(v)\ge\varepsilon n$ for $v\in W$, so each  vertex of $W$ is adjacent to
all vertices of $\overline{V}_{i}''$.
Consequently,
$N_{G}(u)=V(G)\backslash\overline{V}_{i}$ for every $u\in\overline{V}_{i}''$.
\end{proof}

\begin{myCli}\label{2}
$|W|=t-1$.
\end{myCli}

\begin{proof}
By Claim~\ref{1}, $\cup_{i\in[r]}\overline{V}''_{i}$  induces a complete $r$-partite graph, and
every vertex of $W$ is adjacent to all vertices of
$\cup_{i\in[r]}\overline{V}''_{i}$ .
Thus $|W|\le t-1$.

If $t=1$, then $|W|=0$, and the claim holds.
Assume $t\ge2$ and $|W|\le t-2$.
Pick $v_{0}\in\overline{V}_{i}''$, and let $G'$ be the graph obtained from $G$ by
deleting all edges in $\cup_{i\in[r]}E(G[V_{i}\backslash\overline{V}''_{i}])$
and adding all edges between $v_{0}$ and $\overline{V}_{i}''\backslash\{v_{0}\}$.
Then $G'$ is a subgraph of $K_{t-1}\vee T_{rn,r}$, so $G'$ is $F$-free.

On the other hand, by~\eqref{m1}, Lemma~\ref{lem:4.3} and the bound
$x_{\min}^{2}\ge\frac{1}{n}\bigl(1-\frac{1}{\log\log n}\bigr)$, we have
\begin{displaymath}
\begin{split}
q(G')-q(G)
&\geq \sum_{u\in \overline{V}_{i}''\backslash \{v_{0}\}}(x_{u}+x_{v_{0}})^{2}
-\sum_{uv\in \cup_{i\in[r]}E(G[V_{i}\backslash\overline{V}''_{i}])}(x_{u}+x_{v})^2\\
&\geq\frac{4}{n}\Big(1-\frac{1}{\log\log n}\Big)\Big(\frac{n}{r}-2\varepsilon n\Big)-\frac{1024rM_{1}^{2}}{n}\\
&>\frac{3}{r},
\end{split}
\end{displaymath}
contradicting the maximality of $q(G)$.
Hence $|W|=t-1$.
\end{proof}

By the definition of $\overline{V}''_{i},$ $G$ is $F$-free, Claims~\ref{1} and \ref{2}, we have each $\overline{V}_{i}$ is an independent set.
Therefore there exist integers $n_{1},\dots,n_{r}$ with
$\sum_{i=1}^{r}n_{i}=n-t+1$ such that $G$ is a spanning subgraph of
$K_{t-1}\vee K_{r}(n_{1},\dots,n_{r})$.
Finally, Lemma~\ref{t-color} and $x_{\min}^2 \ge \frac{1}{n}\bigl(1 - \frac{1}{\log\log n}\bigr)$ imply
$G\cong K_{t-1}\vee T_{n-t+1,r}$,
completing the proof.
\end{proof}

\section{Proof of Theorem \ref{thm:main1}}

 Claims $1$ and $2$ in the proof of \cite[Theorem $2.3$]{ZLS2025} yield the following result.

\begin{lem}[\cite{ZLS2025}]\label{min-comp}
Let  $F$ be a graph with  $\chi(F) = r+1 \ge 4$.  There exist infinitely many positive integers $n$ such that, for every
$G\in\mathrm{EX}_q(n,F)$ and every nonnegative unit eigenvector
$\mathbf{x}=(x_1,\ldots,x_n)$ corresponding to $q(G)$,
$$x_{\textup{min}}^{2}\geq\frac{1}{n}\Big(1-\frac{1}{\log\log n}\Big).$$
\end{lem}

\begin{lem}[\cite{ZLF2026}]\label{sdf}
Let $G$ be a graph of order $n\geq 3$. If $\mathbf{x}=(x_{1},\ldots,x_{n})$ is a nonnegative unit  eigenvector of $q(G)$ and
 $u$ is a vertex for which $x_{u}=x_{\textup{min}}$, then
$$\frac{q(G-u)}{n-2}\geq \frac{q(G)}{n-1}\bigg(1+\frac{1-nx_u^{2}}{(n-2)(1-x_u^{2})}\bigg)-\frac{1-nx_u^{2}}{(n-2)(1-x_u^{2})}.$$
\end{lem}

\begin{proof}[{\bf Proof of Theorem \ref{thm:main1}}]
By Lemmas \ref{pro-lem} and \ref{min-comp}, there exists a sufficiently large $n_{0}$ such that
$$\mathrm{EX}_{q}(n_{0},F)=\{K_{t-1}\vee T_{n_{0}-t+1,r}\}.$$ Assume that
the result holds for $n-1\geq n_{0}$. We now prove that the result still holds for $n$. Let
$G\in \mathrm{EX}_{q}(n,F)$ and $\mathbf{x}=(x_{1},\ldots,x_{n})$ be a nonnegative unit eigenvector corresponding to $q(G)$. By Lemma \ref{pro-lem}, it suffices to show that $x_{\textup{min}}^{2}\geq\frac{1}{n}(1-\frac{1}{\log\log n})$.
 By Lemma \ref{sdf}, we have
$$\frac{\mathrm{ex}_q (n-1,F)}{n-2}\geq  \frac{\mathrm{ex}_q (n,F)}{n-1}\bigg(1
+\frac{1-nx_{\textup{min}}^{2}}{(n-2)(1-x_{\textup{min}}^{2})}\bigg)
-\frac{1-nx_{\textup{min}}^{2}}{(n-2)(1-x_{\textup{min}}^{2})}.$$
From $r\geq3$ and Lemma \ref{Ktv}, we obtain
$$\mathrm{ex}_q (n,F)\geq q(K_{t-1}\vee T_{n-t+1,r})=\alpha n+(t-1)\beta+O(\frac{1}{n})\geq \frac{5}{4}n,$$
where $\alpha=\frac{2(r-1)}{r}$ and $\beta=\frac{4(r-1)}{r(r-2)}$.
Consequently,
\begin{equation}\label{in2}
\frac{\mathrm{ex}_q (n-1,F)}{\mathrm{ex}_q (n,F)}\geq\frac{n-2}{n-1}\Big(1+\frac{1-nx_{\textup{min}}^{2}}{5(n-2)(1-x_{\textup{min}}^{2})}\Big)
\geq\frac{n-2}{n-1}\Big(1+\frac{1-nx_{\textup{min}}^{2}}{5n}\Big).
\end{equation}

By the induction hypothesis,
$$\mathrm{ex}_q(n-1,F)=q(K_{t-1}\vee T_{n-t,r}).$$
Moreover, since $K_{t-1}\vee T_{n-t+1,r}$ is $F$-free,
\begin{align}\label{in3}
\begin{split}
\frac{\mathrm{ex}_q (n-1,F)}{\mathrm{ex}_q (n,F)}
&=\frac{q(K_{t-1}\vee T_{n-t,r})}{\mathrm{ex}_q (n,F)}\\
&\leq \frac{q(K_{t-1}\vee T_{n-t,r})}{q(K_{t-1}\vee T_{n-t+1,r})}\\
&=\frac{\alpha (n-1)+(t-1)\beta+O(\frac{1}{n})}{\alpha n+(t-1)\beta+O(\frac{1}{n})}\\
&=1-\frac{\alpha}{\alpha n+(t-1)\beta+O(\frac{1}{n})}+O(\frac{1}{n^{2}})\\
&\leq1-\frac{1}{n+2t}+O(\frac{1}{n^{2}}).
\end{split}
\end{align}
Combining (\ref{in2}) and  (\ref{in3}) yields
$$1+\frac{1-nx_{\textup{min}}^{2}}{5n}\leq\Big(1+\frac{1}{n-2}\Big)\Big(1-\frac{1}{n+2t}+O(\frac{1}{n^{2}})\Big)
=1+O(\frac{1}{n^{2}}).$$
Thus
$$x_{\textup{min}}^{2}\geq\frac{1}{n}\Big(1-O(\frac{1}{n})\Big)\geq\frac{1}{n}\Big(1-\frac{1}{\log\log n}\Big).$$
By induction on the number of vertices, we complete the proof of Theorem \ref{thm:main1}.
\end{proof}

\subsection*{Acknowledgement}
 The authors acknowledge the use of AI assistance in discussions on the proof of Lemma~\ref{lem:2.2}. All mathematical arguments and proofs in the final manuscript were checked and written by the authors.

\end{document}